\documentclass[11pt]{amsart}
\usepackage[english]{babel}
\usepackage[T1]{fontenc}
\usepackage[utf8]{inputenc}
\usepackage{amsmath,amssymb,amsthm}
\usepackage{mathtools}
\usepackage{enumitem}
\usepackage[colorlinks=true,linkcolor=blue,citecolor=blue,urlcolor=blue]{hyperref}

\hypersetup{
  colorlinks=true,
  linkcolor=blue,
  citecolor=blue,
  urlcolor=blue
}

\newtheorem{theorem}{Theorem}
\newtheorem{lemma}[theorem]{Lemma}
\newtheorem{definition}[theorem]{Definition}
\newtheorem{corollary}[theorem]{Corollary}
\newtheorem*{remark}{Remark}

\newcommand{\ww}{\omega^\omega}
\newcommand{\wless}{\omega^{<\omega}}
\newcommand{\tw}{2^\omega}
\newcommand{\tless}{2^{<\omega}}
\newcommand{\lh}{\operatorname{lh}}
\newcommand{\dom}{\operatorname{dom}}

\newcommand{\leT}{\leq_T}

\newcommand{\restr}{\upharpoonright}

\title{1-Genericity and Almost Everywhere Domination}

\author{Xuanheng Zhao}
\address{School of Mathematics\\
 Nanjing University\\
 Nanjing, Jiangsu 210093, People's Republic of China}

\email{xuanheng21@gmail.com}

\begin{document}


\begin{abstract}
We prove that there is a
1-generic set which computes a uniformly almost everywhere dominating
function. The proof has two ingredients: a Baire-space
version of the Chong--Downey $\Sigma^0_1$-density criterion \cite{CD} and the forcing of
Cholak--Greenberg--Miller \cite{CGM}.
\end{abstract}

\maketitle

Denis Hirschfeldt raised the following question via a personal email
communication: is there a 1-generic set which computes a uniformly
almost everywhere dominating function? We answer this question in the
affirmative (Corollary \ref{thm:main}).

The result links two phenomena which are usually
studied by different methods. Uniform almost everywhere domination
(u.a.e.d., see Definition~\ref{def:uad}), introduced by Dobrinen and Simpson
\cite{DS} in connection with regularity principles for Lebesgue measure,
is a measure-theoretic strengthening of ordinary domination and hence
imposes substantial computational strength. Simpson \cite{Simpson}
showed that if a set $A$ computes a u.a.e.d. function, then $A'\geq_{tt}\emptyset''$. On the other hand,
Chong and Downey's analysis \cite{CD} of $\Sigma^0_1$-density gives a
structural criterion for being recursive in a 1-generic set. Our theorem
combines these two lines: the Cholak--Greenberg--Miller forcing can be
organized so that the resulting dominating function avoids every r.e.
$\Sigma^0_1$-dense obstruction in Baire space. This shows that
1-genericity is compatible with a strong measure-theoretic domination
property.

Hirschfeldt, Jockusch and Schupp \cite{HJS} introduced a $\{0,\frac{1}{2},1\}$-valued distance on the Turing degrees. They then introduced a division of the Turing degrees into attractive degrees, those that are at distance $\frac{1}{2}$ from measure 1 many degrees; and dispersive degrees, those that are at distance
1 from measure 1 many degrees. It was proved in \cite{HJS} that every weakly 2-generic degree is dispersive, every 1-random degree is attractive, and every a.e.d. degree (it computes a u.a.e.d. function) is attractive. And recently, Hirschfeldt and Royer \cite{HR} showed that if a degree is not DNC, then it is attractive iff it is a.e.d. Hence our result implies that there is an attractive 1-generic degree, which answers the fifth question in \cite[Question 9.2]{HJS} in the negative.

We write $\sigma\sqsubseteq\tau$ if the finite string $\sigma$ is an
initial segment of $\tau$, and $\sigma\sqsubset\tau$ if the containment is
proper. The same notation is used for finite strings and infinite
sequences. If $\sigma\in\wless$, then $\lh(\sigma)$ denotes its length. A
set $G\in\tw$ is \emph{1-generic} if, for every r.e. set $S\subseteq
\tless$, either some initial segment of $G$ lies in $S$, or some initial
segment $\rho\sqsubseteq G$ has no extension in $S$.

\begin{definition}[Baire-space $\Sigma^0_1$-density]
Let $f\in\ww$ and let $W\subseteq\wless$. We say that $W$ is
\emph{dense in $f$} if
\[
   (\forall n)(\exists\sigma\in W)\bigl(f\restr n\sqsubseteq\sigma\bigr).
\]
An infinite r.e. set $Y\subseteq\wless$ is \emph{$\Sigma^0_1$-dense in
$f$} if
\begin{enumerate}[label=(\alph*)]
\item no member of $Y$ is an initial segment of $f$;
\item for every r.e. set $W\subseteq\wless$ which is dense in $f$, some
      member of $W$ extends some member of $Y$.
\end{enumerate}
The corresponding definition in Cantor space is the same, with
$\tless$ in place of $\wless$.
\end{definition}

\begin{theorem}[Chong--Downey \cite{CD}]\label{thm:CD}
Let $A\in\tw$. If there is no infinite r.e. set $Y\subseteq\tless$ which
is $\Sigma^0_1$-dense in $A$, then $A$ is recursive in a 1-generic set.
\end{theorem}

The proof below supplies the Baire-space form needed later. Define a
recursive block coding $c:\wless\to\tless$ by
\[
  c(a_0\ldots a_{k-1})
  =1^{a_0}0\,1^{a_1}0\cdots 1^{a_{k-1}}0.
\]
For $f\in\ww$, let
\[
  C(f)=1^{f(0)}0\,1^{f(1)}0\,1^{f(2)}0\cdots\in\tw.
\]
Then $c(\sigma)\sqsubseteq C(f)$ if and only if $\sigma\sqsubseteq f$,
and $f\leT C(f)$ by uniformly reading the lengths of successive blocks of
$1$'s separated by $0$'s.

\begin{lemma}\label{lem:transfer}
Let $f\in\ww$. If $Y\subseteq\tless$ is an infinite r.e. set which is
$\Sigma^0_1$-dense in $C(f)$, then there is an infinite r.e. set
$\widehat Y\subseteq\wless$ which is $\Sigma^0_1$-dense in $f$.
\end{lemma}

\begin{proof}
Given $Y$, define
\[
   \widehat Y=\{\sigma\in\wless:(\exists\eta\in Y)\;\eta\sqsubseteq c(\sigma)\}.
\]
This set is r.e.: when an element $\eta$ enters $Y$, search through all
finite Baire strings $\sigma$ and enumerate those satisfying
$\eta\sqsubseteq c(\sigma)$.

We first show that $\widehat Y$ is infinite. For each $k$, let
\[
   V_k=\{c(\sigma):\sigma\in\wless\text{ and }\lh(\sigma)\geq k\}.
\]
Then $V_k$ is r.e. and dense in $C(f)$. Since $Y$ is
$\Sigma^0_1$-dense in $C(f)$, there are $\eta\in Y$ and
$c(\sigma)\in V_k$ such that $\eta\sqsubseteq c(\sigma)$. Thus
$\sigma\in\widehat Y$ and $\lh(\sigma)\geq k$. Since this holds for every
$k$, the set $\widehat Y$ has strings of arbitrarily large length, and
therefore is infinite.

We next check condition (a). Suppose, toward a contradiction, that
$\sigma\in\widehat Y$ and $\sigma\sqsubseteq f$. Choose $\eta\in Y$ with
$\eta\sqsubseteq c(\sigma)$. Since $\sigma\sqsubseteq f$, we have
$c(\sigma)\sqsubseteq C(f)$, and hence $\eta\sqsubseteq C(f)$. This
contradicts condition (a) for $Y$ as a $\Sigma^0_1$-dense set in $C(f)$.

Now let $W\subseteq\wless$ be r.e. and dense in $f$. Put
\[
   c(W)=\{c(\sigma):\sigma\in W\}\subseteq\tless.
\]
Then $c(W)$ is r.e. and dense in $C(f)$: if $\rho\sqsubseteq C(f)$ is
finite, choose $n$ large enough so that $\rho\sqsubseteq c(f\restr n)$;
since $W$ is dense in $f$, there is $\sigma\in W$ extending $f\restr n$,
and then $c(\sigma)$ extends $\rho$. Since $Y$ is $\Sigma^0_1$-dense in
$C(f)$, there are $\eta\in Y$ and $\sigma\in W$ such that
$\eta\sqsubseteq c(\sigma)$. Hence $\sigma\in W\cap\widehat Y$, so some
member of $W$ extends some member of $\widehat Y$. This proves condition
(b).
\end{proof}

\begin{theorem}[Baire-space Chong--Downey criterion]\label{thm:baireCD}
Let $f\in\ww$. If there is no infinite r.e. set
$Y\subseteq\wless$ which is $\Sigma^0_1$-dense in $f$, then $f$ is
recursive in a 1-generic set.
\end{theorem}

\begin{proof}
Assume that no infinite r.e. subset of $\wless$ is $\Sigma^0_1$-dense in
$f$. By Lemma~\ref{lem:transfer}, there is no infinite r.e. subset of
$\tless$ which is $\Sigma^0_1$-dense in $C(f)$. Applying
Theorem~\ref{thm:CD}, choose a 1-generic set $G\in\tw$ such that
$C(f)\leT G$. Since $f\leT C(f)$, we have $f\leT G$.
\end{proof}

\begin{definition}\label{def:uad}
For $f,g\in\ww$, say that $f$ \emph{dominates} $g$, and write
$g\leq^* f$, if
\[
   (\exists n_0)(\forall n\geq n_0)\; g(n)\leq f(n).
\]
A function $f\in\ww$ is \emph{uniformly almost everywhere dominating (u.a.e.d. for short)} if
\[
 \mu\bigl(\{Z\in\tw:(\forall g\in\ww)(g\leT Z\Rightarrow g\leq^* f)\}\bigr)=1.
\]
\end{definition}

We use the forcing of Cholak, Greenberg, and Miller \cite{CGM}. First, we fix a partial recursive functional
$\Phi$ such that the following
implication holds:
\[
  \mu\bigl(\{Z\in\tw:\Phi^Z\text{ is total and } f \geq^* \Phi^Z\}\bigr)
  =\mu(\dom\Phi)
\]
implies that $f$ is u.a.e.d., where
\[
   \dom\Phi=\{Z\in\tw:\Phi^Z\text{ is total}\}.
\]
Such a functional is obtained by coding each partial
recursive functional into a cylinder \cite[Lemma~1.7]{CGM}.

We assume the usual stage convention: if $\Phi^Z(n)[s]\downarrow$, then
the value is $<s$, and if $\Phi^Z(n)[s]\downarrow$, then
$\Phi^Z(m)[s]\downarrow$ for all $m<n$.

For $n\in\omega$, let
\[
   D_n=\{Z\in\tw:(\forall k<n)\;\Phi^Z(k)\downarrow\}.
\]
For $\tau\in\wless$ and $m\leq n\leq\lh(\tau)$, let
\[
   D_{[m,n)}[\tau]
   =\{Z\in\tw:(\forall k\in[m,n))\;\Phi^Z(k)[\tau(k)]\downarrow<\tau(k)\}.
\]
If $m=n$, then $D_{[m,n)}[\tau]=2^\omega$.

\begin{definition}[CGM forcing \cite{CGM}]\label{def:forcing}
A condition is a pair $p=(\sigma_p,\varepsilon_p)$, where
$\sigma_p\in\wless$ and $\varepsilon_p\in\mathbb Q^+$. We write
$n_p=\lh(\sigma_p)$. For conditions $p,q$, define $q\leq p$ if
\begin{enumerate}[label=(\roman*)]
\item $\sigma_p\sqsubseteq\sigma_q$;
\item $\varepsilon_q\leq\varepsilon_p$;
\item if $\sigma_p\sqsubset\sigma_q$, then
\[
   \mu\bigl(\dom\Phi\setminus D_{[n_p,n_q)}[\sigma_q]\bigr)
   +\varepsilon_q \le \varepsilon_p.
\]
\end{enumerate}
\end{definition}

The following Lemmas~\ref{lem:transitive}--\ref{lem:genericdom} are basic
properties of the CGM forcing proved in \cite{CGM} and needed below.

\begin{lemma}\label{lem:transitive}
The relation $\leq$ is transitive.
\end{lemma}

\begin{lemma}\label{lem:length}
For every condition $p$ and every $N$, there is a condition $q\leq p$
such that $n_q>N$.
\end{lemma}

\begin{lemma}\label{lem:promise}
For every condition $p$ and every rational $\eta>0$, there is a condition
$q\leq p$ such that $\varepsilon_q<\eta$.
\end{lemma}

\begin{lemma}\label{lem:cgmdense}
Let $p$ be a condition. There are an r.e. set $S\subseteq\wless$, a
condition $p^*\leq p$, and a number $N$ such that:
\begin{enumerate}[label=(\roman*)]
\item for every $\tau\in S$, there is a condition $q\leq p$ with
      $\sigma_q=\tau$;
\item for every condition $q\leq p^*$, if $n_q>N$, then $\sigma_q\in S$.
\end{enumerate}
\end{lemma}

\begin{remark}\label{rem:lemma10}
This is the explicit form of \cite[Lemma 4.7]{CGM} which follows from the proof given there.
\end{remark}

\begin{lemma}\label{lem:genericdom}
Let $(p_s)_{s\in\omega}$ be a descending sequence of conditions such that
\[
   (\forall N)(\exists s)\; n_{p_s}>N
   \quad\text{and}\quad
   (\forall k)(\exists s)\; \varepsilon_{p_s}<2^{-k}.
\]
Let
\[
   f=\bigcup_s\sigma_{p_s}.
\]
Then $f\in\ww$ and $f$ is u.a.e.d.
\end{lemma}

Fix an effective enumeration $(W_e)_{e\in\omega}$ of all r.e. subsets of
$\wless$.

\begin{theorem}\label{thm:avoid}
There is a u.a.e.d. function $f\in\ww$ such
that no infinite r.e. set $Y\subseteq\wless$ is $\Sigma^0_1$-dense in $f$.
\end{theorem}

\begin{proof}
We construct a descending sequence of conditions
\[
   p_0\geq p_1\geq p_2\geq\cdots
\]
in the forcing of Definition~\ref{def:forcing}. Start with
$p_0=(\emptyset,1)$.

At the beginning of cycle $e$, assume that $p_{3e}$ has been chosen.

\smallskip
\noindent\emph{Length substage.}
By Lemma~\ref{lem:length}, choose $p_{3e+1}\leq p_{3e}$ such that
$n_{p_{3e+1}}>e$.

\smallskip
\noindent\emph{Promise substage.}
By Lemma~\ref{lem:promise}, choose $p_{3e+2}\leq p_{3e+1}$ such that
$\varepsilon_{p_{3e+2}}<2^{-e}$.

\smallskip
\noindent\emph{$\Sigma^0_1$-density substage for $W_e$.}
Put $p=p_{3e+2}$. Apply Lemma~\ref{lem:cgmdense} to this $p$, and obtain
a r.e. set $S_e$, a condition $p_e^*\leq p$, and a number $N_e$
satisfying clauses (i) and (ii) of that lemma. Now distinguish two cases.

\smallskip
\noindent\emph{Case 1.}
There are $\sigma\in W_e$ and a positive rational $\delta$ such that
$(\sigma,\delta)\leq p$. Choose one such pair and put
\[
   p_{3e+3}=(\sigma,\delta).
\]

\smallskip
\noindent\emph{Case 2.}
No such pair exists. Put
\[
   p_{3e+3}=p_e^*.
\]
This completes the construction.

Let
\[
   f=\bigcup_s\sigma_{p_s}.
\]
The length and promise substages ensure the hypotheses of
Lemma~\ref{lem:genericdom}; hence $f$ is u.a.e.d.

It remains to verify that no r.e. set is $\Sigma^0_1$-dense in $f$. Fix
$e$. We show that $W_e$ fails one of the two defining clauses of
$\Sigma^0_1$-density in $f$.

If Case 1 occurs in cycle $e$, $W_e$ fails clause (a).

Suppose Case 2 occurs. We may assume that no member of $W_e$
is an initial segment of $f$. We claim that the r.e. set $S_e$ is dense in
$f$ and that no member of $S_e$ extends a member of $W_e$. This will show
that $W_e$ fails (b).

First, $S_e$ is dense in $f$. Let $m$ be arbitrary. Since the lengths of
the stems in the constructed sequence are unbounded, choose
$t>3e+3$ such that
\[
   p_t\leq p_{3e+3}=p_e^*
   \quad\text{and}\quad
   n_{p_t}>\max\{m,N_e\}.
\]
By Lemma~\ref{lem:cgmdense}(ii), $\sigma_{p_t}\in S_e$. Since
$\sigma_{p_t}\sqsubseteq f$ and $\lh(\sigma_{p_t})>m$, the string
$\sigma_{p_t}$ extends $f\restr m$. Thus $S_e$ is dense in $f$.

Second, no member of $S_e$ extends a member of $W_e$. Suppose otherwise
that
\[
   \sigma\in W_e,\qquad \tau\in S_e,\qquad \sigma\sqsubseteq\tau.
\]
By Lemma~\ref{lem:cgmdense}(i), choose a condition $q\leq p$ with
$\sigma_q=\tau$. Since $q\leq p$, we have $\sigma_p\sqsubseteq\tau$.
Since both $\sigma$ and $\sigma_p$ are initial segments of $\tau$, they
are comparable. If $\sigma\sqsubseteq\sigma_p$, then
$\sigma_p\sqsubseteq f$, and hence $\sigma\sqsubseteq f$, contrary to our
standing assumption in Case 2. Therefore $\sigma_p\sqsubset\sigma$.

Let $\varepsilon_q$ be the promise of $q$. From $q\leq p$ we have
$\varepsilon_q\leq\varepsilon_p$. Also, since $\sigma_p\sqsubseteq\tau$
and in fact $\sigma_p\sqsubset\tau$, the definition of $q\leq p$ gives
\[
   \mu\bigl(\dom\Phi\setminus D_{[n_p,\lh(\tau))}[\tau]\bigr)
   +\varepsilon_q \le \varepsilon_p.
\]
Moreover, $\sigma_p\sqsubset\sigma\sqsubseteq\tau$ implies
\[
   D_{[n_p,\lh(\tau))}[\tau]
   \subseteq D_{[n_p,\lh(\sigma))}[\sigma].
\]
Therefore
\[
\begin{split}
   \mu\bigl(\dom\Phi\setminus D_{[n_p,\lh(\sigma))}[\sigma]\bigr)
   +\varepsilon_q
   &\leq
   \mu\bigl(\dom\Phi\setminus D_{[n_p,\lh(\tau))}[\tau]\bigr)
   +\varepsilon_q \\
   &\le \varepsilon_p.
\end{split}
\]
Together with $\sigma_p\sqsubset\sigma$ and
$\varepsilon_q\leq\varepsilon_p$, this shows that $(\sigma,
\varepsilon_q)\leq p$. This contradicts Case 2. Hence no member of $S_e$
extends a member of $W_e$.
\end{proof}

\begin{corollary}\label{thm:main}
There is a 1-generic set which computes a u.a.e.d. function.
\end{corollary}

\begin{proof}
Let $f$ be given by Theorem~\ref{thm:avoid}. By that theorem, no infinite
r.e. set of finite strings is $\Sigma^0_1$-dense in $f$. By
Theorem~\ref{thm:baireCD}, there is a 1-generic set $G\in\tw$ such that
$f\leT G$.
\end{proof}

\section*{Acknowledgements}
The author was supported by Focused Research Grants of Natural Science Foundation of Jiangsu Province No. BK20243060. The author thank Denis Hirschfeldt for helpful suggestions.

\end{document}